\documentclass[a4paper,12pt,reqno]{amsart}
\usepackage{amsfonts}
\usepackage{amssymb}
\usepackage{amsxtra}
\usepackage[T1]{fontenc}
\usepackage[utf8]{inputenc}
\usepackage[libertinus,vvarbb]{newtx}
\usepackage[margin=1in]{geometry}
\usepackage{enumitem}
\usepackage{graphicx,color}
\usepackage{xparse}
\usepackage{xurl}
\usepackage[bookmarksdepth=2]{hyperref}

\usepackage[normalem]{ulem}

\newcommand{\stkout}[1]{\ifmmode\text{\sout{\ensuremath{#1}}}\else\sout{#1}\fi}
\newcommand{\crsout}[1]{\ifmmode\text{\xout{\ensuremath{#1}}}\else\xout{#1}\fi}

\numberwithin{equation}{section}

\newtheorem{theorem}{Theorem}[section]

\newtheorem{corollary}[theorem]{Corollary}
\newtheorem{proposition}[theorem]{Proposition}

\theoremstyle{definition}

\newtheorem{remark}[theorem]{Remark}

\newcommand{\ind}{\mathbb{1}}

\newcommand{\C}{\mathbb{C}}
\newcommand{\R}{\mathbb{R}}

\newcommand{\N}{\mathbb{N}}
\newcommand{\sph}{\mathbb{S}}

\newcommand{\mH}{\mathcal{H}}
\newcommand{\mM}{\mathcal{M}}
\newcommand{\mA}{\mathcal{A}}
\newcommand{\mS}{\mathcal{S}}
\newcommand{\mF}{\mathcal{F}}
\newcommand{\mD}{\mathcal{D}}
\newcommand{\tm}{\tilde{m}}

\newcommand{\norm}[1]{\left\lVert #1 \right\rVert}
\newcommand{\abs}[1]{\left\lvert #1 \right\rvert}

\mathchardef\mathhyphen="2D

\renewcommand{\le}{\leqslant}
\renewcommand{\ge}{\geqslant}

\NewDocumentCommand{\formula}{ssom}{%
 \IfBooleanTF{#1}{%
  \IfBooleanTF{#2}{%
   \IfValueTF{#3}%
    {\begin{align}\label{#3}\begin{gathered}#4\end{gathered}\end{align}}%
    {\begin{gather}#4\end{gather}}%
  }{%
   \IfValueTF{#3}%
    {\begin{align}\label{#3}\begin{aligned}#4\end{aligned}\end{align}}%
    {\begin{gather*}#4\end{gather*}}%
  }%
 }{%
  \IfValueTF{#3}%
   {\begin{align}\label{#3}#4\end{align}}%
   {\begin{align*}#4\end{align*}}%
 }%
}

\title[Dimension-free control of truncated Riesz transforms by Riesz transforms]{On the dimension-free control of higher order truncated Riesz transforms by higher order Riesz transforms}

\author{Maciej Kucharski}
\address{Maciej Kucharski\\
		Institute of Mathematics\\
	University of Wrocław\\
	Plac Grun\-waldzki 2\\
	50-384 Wrocław\\
	Poland}
\email{maciej.kucharski@math.uni.wroc.pl}

\author{Mateusz Kwaśnicki}
\address{Mateusz Kwaśnicki \\ Department of Analysis and Stochastic Processes \\ Wrocław University of Science and Technology \\ ul. Wybrzeże Wyspiańskiego 27 \\ 50-370 Wrocław, Poland}
\email{mateusz.kwasnicki@pwr.edu.pl}

\author{B{\l}a{\.z}ej Wr{\'o}bel}
\address{B{\l}a{\.z}ej Wr{\'o}bel\\
		Institute of Mathematics \\
			Polish Academy of Sciences\\
			\'Sniadeckich 8\\
			00–656 Warszawa\\
	Poland \&
	Institute of Mathematics\\
	University of Wrocław\\
	Plac Grun\-waldzki 2\\
	50-384 Wrocław\\
	Poland}
\email{blazej.wrobel@math.uni.wroc.pl}

\begin{document}

\begin{abstract}
Fix a positive integer $k$. Let $R_k$ be a higher order Riesz transform of order $k$ on $\R^d$ and let $R_k^t,$ $t>0,$ be the corresponding truncated Riesz transform. We study the relation between  $\|R_k f\|_{L^p(\R^d)}$ and $\|R_k^t f\|_{L^p(\R^d)}$ for $p=1$, $p=\infty,$ and $p=2.$ We do this by analyzing the 
factorization operator $M_k^t$ defined by the relation $R_k^t=M_k^t R_k.$ The operator $M_k^t$ is a convolution operator associated with an $L^1$ radial kernel $b_{k,d}^t(x)=t^{-d}b_{k,d}(x/t),$ where $b_{k,d}(x):=b_{k,d}^1(x).$

We prove that  $b_{k,d} \ge 0$ only for $k=1,2.$ We also show that for fixed $k\ge 3$, 
\[
\lim_{d\to \infty}\|b_{k,d}\|_{L^1(\R^d)}=\infty.
\]
 This  contrasts with the cases $k=1,2$, where it is known that $\|b_{k,d}\|_{L^1(\R^d)}=1$. Finally, we show that for any positive integer $k$, the Fourier transform of $b_{k,d}$ is bounded in absolute value by $1.$ This implies the contractive estimate
\[
\|R_k^t f\|_{L^2(\R^d)}\le \|R_k f\|_{L^2(\R^d)}
\]
and an analogous estimate for general singular integrals with smooth kernels for radial input functions $f.$

\end{abstract}
\maketitle

\section{Introduction}

Let $k$ be a positive integer and denote by $\mH_k = \mH_k^d$ the space of spherical harmonics of degree $k$ on the Euclidean sphere $\mathbb{S}^{d-1}.$ We identify $P \in \mH_k$ with the corresponding harmonic polynomial, which is homogeneous of degree $k.$ Consider the kernel
\begin{equation}
	\label{eq:KP}
	K_P(x) = \gamma_{k,d} \frac{P(x)}{|x|^{d+k}} \qquad\textrm{ with } \qquad  \gamma_{k,d} = \frac{\Gamma(\frac{k+d}{2})}{\pi^{d/2}\Gamma(\frac{k}{2})}.
\end{equation}
The higher order Riesz transform $R_P$  of order $k$ corresponding to $P$ is defined by  
\begin{equation} \label{eq:R}
	R_P f(x)= \lim_{t \to 0^+} R_P^t f(x), \qquad\textrm{ where } \qquad R_P^t f(x) = \gamma_{k,d} \int_{|y|>t} \frac{P(y)}{|y|^{d+k}} f(x-y) \, dy.
\end{equation}
The operators $R_P^t,$ $t>0$, are called truncated Riesz transforms. It is well known, see \cite[p. 73]{stein}, that the Fourier multiplier associated with the Riesz transform $R_P$ equals 
\begin{equation} \label{eq:m}
	\rho_P(\xi) = (-i)^k P\left(\tfrac{\xi}{|\xi|}\right),\qquad \xi\in \R^d.
\end{equation}

Most of the  time, the specific spherical harmonic  is not important  for our considerations. In such cases, we write $R_k$ to denote a higher order Riesz transform of order $k$ corresponding to some $P\in \mH_k.$  A similar convention applies to the truncated Riesz transform, which we denote by $R_k^t.$ For future reference, we also define the maximal truncated Riesz transform by
\[
R_k^*f(x)=\sup_{t>0}|R_k^t f(x)|.
\]

It is known that the truncated Riesz transform of order $k$ factors according to
\begin{equation}
\label{eq:fact}
R_k^t=M^t_k R_k.
\end{equation}
The factorization operator $M^t_k$ above is a convolution operator which is bounded on all $L^p(\R^d)$ spaces for $p\in [1,\infty].$ We denote by $b_{k,d}^t$ the convolution kernel of this operator which is known to be radial, real valued, and to belong to $L^1(\R^d)$. For general  $k$, the factorization is  implicit in \cite[Section 2]{mateu_verdera} ($k=1$), \cite[Section 2]{mov1} ($k$ even), and \cite[Section 4]{mopv} ($k$ odd). In the case $k=1$ the factorization \eqref{eq:fact} is given explicitly in \cite{kw} and \cite{LiuMeZhu}. For general positive integers $k$ this is justified in \cite[Proposition 2.1]{kwz}, whose proof also shows that  $b_{k,d}^t(x)=t^{-d}b_{k,d}^1(x/t).$ 

It is clear from \eqref{eq:fact} that the operator $M_k:=M_k^1$ and its kernel $b_{k}:=b_{k,d}^1$ provide important information about the relation between $R_k$ and $R_k^t$ or $R_k^*.$ For instance, when $k$ is even,  the work of Mateu, Orobitg, and Verdera \cite[Section 2]{mov1} implies that 
\[
b_k(x)=P_{k,d}(|x|^2)\ind{_B}(x),
\]
where $P_{k,d}$ is a polynomial of degree $k/2-1$ and $\ind{_B}$ denotes the indicator function of the Euclidean unit ball $B$ in $\R^d.$ This implies the estimate
\[
\abs{b_k(x)}\le C_{k,d}\ind{_B}(x),
\]
where $C_{k,d}$ is a constant,
and leads to
\begin{equation}
\label{eq:HL pointwise even}
|R_k^*f(x)|\le C_{k,d}\mM(R_k f)(x),
\end{equation}
where $\mM$ denotes the centered Hardy--Littlewood maximal operator over Euclidean balls. The estimate \eqref{eq:HL pointwise even}  can be viewed as an improved version of Cotlar's classical inequality, which for higher order Riesz transforms asserts that
\begin{equation*}
|R_k^*f(x)|\le B_{k,d}(\mM(R_k f)(x)+ \mM(f)(x)),
\end{equation*}
where $B_{k,d}$ is a constant. In particular, \eqref{eq:HL pointwise even} implies the following $L^p$ inequality:
\begin{equation}
\label{eq:LpRmaxR}
\|R_k^*f(x)\|_{L^p(\R^d)}\le C_{p,k,d} \|R_kf(x)\|_{L^p(\R^d)},
\end{equation}
valid for $p\in (1,\infty].$ When $k$ is odd,  a weaker version of \eqref{eq:HL pointwise even} holds, involving the composition of the Hardy--Littlewood maximal operator
\begin{equation}
\label{eq:HL pointwise odd}
|R_k^*f(x)|\le C_{k,d}\mathcal (\mM\circ \mM)(R_k f)(x).
\end{equation}
The above inequality was obtained by Mateu, Orobitg, Per\'ez, and Verdera in \cite[Section 4]{mopv}. This implies that \eqref{eq:LpRmaxR} remains valid for all positive integers $k.$ However, the order of growth of the constant $C_{p,k,d}$ coming from the proofs in \cite{mov1} and \cite{mopv} is exponential in the dimension $d.$

Recently, the first and  third  authors, in collaboration with Zienkiewicz \cite{kwz}, proved, among other results, that for fixed $k$ one may take a dimension-free constant (independent of $d$) in \eqref{eq:LpRmaxR}. They also established explicit dimension-free estimates in terms of $p.$ To achieve this, they employed various techniques from the theory of singular integrals, centered around  both real and complex methods of rotations.

Interestingly,  for $k=1$ and $k=2$, one can obtain a dimension-free variant of \eqref{eq:LpRmaxR} more directly. For $k=2$,  $b_2=\frac{1}{|B|}\ind_{B},$ where by $|B|$ we denote the Lebesgue measure of the ball. Therefore the maximal function corresponding to $M_2$ equals $\mM$ --- the Hardy--Littlewood maximal operator, see e.g.\ \cite[p.\ 427]{verdera conf}. Thus, one may take $1$ as the constant $C_{k,d}$ in \eqref{eq:HL pointwise even}, and a dimension-free variant of  \eqref{eq:LpRmaxR} follows from the classical work of Stein and Str\"omberg \cite{SteinMax}, \cite{StStr}. Somewhat surprisingly,  the kernel $b_1$ also turns out to be non-negative for $k = 1$. This was proved by Liu, Melentijević, and Zhu \cite{LiuMeZhu} and was an important ingredient there to obtain an improved variant of \eqref{eq:LpRmaxR} with the explicit constant $ (2+\frac{1}{\sqrt{2}})^{2/p}$ replacing $C_{p,k,d}$ for $p\ge 2.$ A consequence of \cite{LiuMeZhu} is also the pointwise bound
\begin{equation}
\label{eq: R1mS}
|R_1^*f(x)|\le \mS(R_1 f)(x),
\end{equation}
where $\mS$ is the spherical maximal operator
\[
\mS f(x)=\sup_{r \in (0, \infty)} \frac{1}{|\sph^{d - 1}|} \int_{\sph^{d - 1}} |f|(x + r u) \, du.
\]

A natural question that arises is whether similar properties of $b_k$ hold for $k\ge 3.$ Our first main result states that this is not the case.

\begin{theorem}
    \label{thm: bkL1}
    The kernel $b_k$ of the factorization operator $M_k$ is non-negative only for $k=1,2.$ Furthermore, for every fixed positive integer $k\ge 3$ we have
    \begin{equation}
\label{eq: bkL1 lim}
\lim_{d\to \infty}\|b_k\|_{L^1(\R^d)}=\infty.
    \end{equation}
\end{theorem}

As a corollary of this theorem we prove that for $k\ge 3$ it is impossible to justify a variant of \eqref{eq:HL pointwise even}, \eqref{eq:HL pointwise odd}, or \eqref{eq: R1mS}, which will involve a dimension-free constant. This is the case even if we relax the maximal operators on the left-hand sides of \eqref{eq:HL pointwise even}, \eqref{eq:HL pointwise odd}, \eqref{eq: R1mS} to the single truncation $R_k^1.$ Below, for each dimension $d$, we let $\mA_d$ be a non-negative sublinear operator  that is a contraction on $L^{\infty}(\R^d)$ and is bounded on $L^2(\R^d).$ More precisely, we assume that for any $f,g\in L^2(\R^d)+L^{\infty}(\R^d)$ and $\lambda\in \C$, the following relations
\begin{equation*}
\mA_d(f)(x)\ge 0,\quad \mA_d(\lambda f)(x)=|\lambda| \mA_d f(x),\quad \mA_d(f+g)(x)\le \mA_d(f)(x)+\mA_d(g)(x)
\end{equation*}
hold for a.e.\ $x\in \R^d.$ We also impose that
\[
\|\mA_d(f)\|_{L^{\infty}(\R^d)}\le \|f\|_{L^{\infty}(\R^d)},\qquad f\in L^{\infty}(\R^d)
\]
and that there is a constant $C>0$ such that
\[
\|\mA_d(f)\|_{L^{2}(\R^d)}\le C\|f\|_{L^{2}(\R^d)},\qquad f\in L^{2}(\R^d).
\]
Notice that these assumptions imply that $\mA_d$ is continuous on $L^2(\R^d).$ Particular examples of such operators $\mA_d$ are $\mM,$ $\mM\circ \mM$, the spherical maximal operator $\mathcal S$ in dimensions $d\ge 3$, or any composition of such operators.
\begin{corollary}
\label{cor: bkL1}
For each $d$ let $\mA_d$ be a non-negative sublinear operator which is bounded on $L^2(\R^d)$ and is a contraction on $L^{\infty}(\R^d).$ Fix $k\ge 3$ and assume that there is a constant $C(k,d)$ for which
\[
|R_k^1 f(x)|\le C(k,d) \mA_d(R_k f)(x)
\]
holds for all Schwartz functions $f$ on $\R^d$ and all $x\in \R^d.$ Then $C(k,d)\to \infty$ as $d\to \infty.$
\end{corollary}

We now turn to $L^2$ estimates. Our second main result concerns the Fourier transform
\[\widehat{b}_k(\xi)=\int_{\R^d}b_k(x)\exp(-2\pi i x\cdot \xi)\,dx.\]
This is a radial function, namely we have
\[ \widehat{b}_k(\xi) = m_k(\lvert \xi \rvert) \]
for some function $m_k$, called the \emph{radial profile} of $\widehat{b}_k$.

\begin{theorem}
    \label{thm: fbk}
    For each positive integer $k$ the Fourier transform $\widehat{b}_k$ satisfies $|\widehat{b}_k(\xi)|\le 1,$ $\xi \in \R^d.$ Consequently, the operator $M_k$ is a contraction on $L^2(\R^d)$ and for all $t>0$ we have
\begin{equation}
\label{eq: RktRL2}
\|R_k^t f\|_{L^2(\R^d)}\le \|R_k f\|_{L^2(\R^d)}.
\end{equation}
\end{theorem}

When we restrict the input functions $f$ to radial ones, then \eqref{eq: RktRL2} from Theorem \ref{thm: fbk} may be extended to all singular integrals with smooth kernels. Namely, let  $\Omega\colon \mathbb{S}^d\to \C$ be a smooth function on the unit sphere with integral zero. Define the truncated singular integral $T^t_{\Omega},$ $t>0,$ and the singular integral associated with $\Omega$ by
\[T_{\Omega}^tf(x)= \int_{|y|>t} \frac{\Omega(y/|y|)}{|y|^{d}} f(x-y) \, dy,\qquad T_{\Omega}f(x)=\lim_{t\to 0^+}T_{\Omega}^t f(x).\]

\begin{corollary}
\label{cor: gsir}
Let $\Omega\colon \mathbb{S}^d\to \C$ be a smooth function on the unit sphere with integral zero. Then, for any radial function $f\in L^2(\R^d)$ and all $t>0$ we have
\begin{equation}
\label{eq: TktTL2}
\|T_{\Omega}^tf\|_{L^2(\R^d)}\le \|T_{\Omega} f\|_{L^2(\R^d)}.
\end{equation}
\end{corollary}
\begin{remark}
    A strengthening of \eqref{eq: TktTL2} of the form
    \begin{equation}
    \label{eq: TktTL2 gen}
\|\sup_{t>0}|T_{\Omega}^tf|\|_{L^2(\R^d)}\le C_{\Omega,d}\|T_{\Omega} f\|_{L^2(\R^d)},
\end{equation}
valid for all functions $f\in L^2(\R^d),$ is false even if we allow a constant $ C_{\Omega,d}$ depending on the kernel $\Omega$ and the dimension $d.$ Such an inequality holds if and only if $\Omega$ satisfies an algebraic condition related to its expansion in spherical harmonics, see \cite[Theorem (iv)]{mov1} ($k$ even), and \cite[Theorem 1 (iv)]{mopv} ($k$ odd). This condition is satisfied when $\Omega(x)=P(x/|x|),$ $P\in\mH_k,$ is the kernel of a higher order Riesz transform. In such cases, it follows from \cite{kwz} that  \eqref{eq: TktTL2 gen} holds with a constant depending  on $k$ but independent of $P\in \mH_k$ and of the dimension $d.$  
\end{remark}

\subsection{Overview of our methods and the structure of the paper}

In Section \ref{sec: form} we give useful formulas for the radial profile $m_k$ of the multiplier $\widehat{b}_k,$ see Proposition \ref{pro: mk}, and for the radial profile $B_k$ of the kernel $b_k,$ see Proposition \ref{pro: B_k}. The proof of Proposition \ref{pro: mk} is based on Bochner's relation. Proposition \ref{pro: B_k} is derived from Proposition \ref{pro: mk} by an integration by parts argument similar to the one used in  \cite[Appendix 4.1]{LiuMeZhu}.

Section \ref{sec: kere} is devoted to kernel estimates. First, we justify Theorem \ref{thm: bkL1}. The proof is based on Proposition \ref{pro: B_k} and the considerations are split between $k$ odd and $k$ even. The odd case is easier because the kernel $b_k(x)$ does not vanish for $|x|\ge 1$ and, moreover, $\int_{|x|>1}|b_k(x)|\,dx$ goes to infinity with the dimension. The analysis in the even case is more elaborate, because then $b_k(x)$ vanishes for $|x|\ge 1$. However, Proposition \ref{pro: B_k} implies that $B_k(r)$ is a polynomial. Then, the change of variables $r^d=e^{-s}$, followed by a more careful analysis of $B_k(e^{-s/d})$, reduces  the problem to an estimate involving a Laguerre polynomial of degree $k/2-1$, see \eqref{eq: bkLag}. We finish Section \ref{sec: kere} with a proof of Corollary \ref{cor: bkL1}.

Section \ref{sec: L2e} is devoted to the  proof of the $L^2$ results: Theorem \ref{thm: fbk} and Corollary \ref{cor: gsir}. The proof of Theorem \ref{thm: fbk} is based on the formula \eqref{eq: mk 1} from Proposition \ref{pro: mk} for $m_k$ together with oscillatory estimates from \cite{lms}. Corollary \ref{cor: gsir} follows by using the decomposition of a general kernel $\Omega$ on the sphere into spherical harmonics.

\subsection{Notation}

\begin{enumerate}

\item  Positive integers $d$ and $k$ denote the dimension of the Euclidean space $\R^d$ and the order of the Riesz transform, respectively. 
\item For a  positive integer $\ell$, we let  $\mF_{\ell}$ be the $\ell$-dimensional Fourier transform
\[
\mF_{\ell}(f)(\xi)=\int_{\R^{\ell}}f(x)\exp(-2\pi i x\cdot \xi)\,dx.
\]
When $\ell=d$ we abbreviate $\mF_{\ell}(f)=\widehat{f}.$

\item We denote by $J_{\nu}$ the Bessel function of the first kind and order $\nu,$ i.e.\
\[
J_{\nu}(x) = \sum_{n=0}^{\infty} \frac{(-1)^n}{n! \, \Gamma(n + \nu + 1)} \left( \frac{x}{2} \right)^{2n + \nu}.
\]

\item The symbol ${_p}F_q$ represents the generalized hypergeometric function defined by
\[
{_p}F_q(a_1, \ldots, a_p; b_1, \ldots, b_q; z)
= \sum_{n=0}^{\infty} \frac{(a_1)^{\overline{n}} \cdots (a_p)^{\overline{n}}}{(b_1)^{\overline{n}} \cdots (b_q)^{\overline{n}}} \cdot \frac{z^n}{n!},
\]
where  $a^{\overline{n}} = a (a + 1) \dots (a + n - 1)$ is the  rising factorial (Pochhammer symbol). In this paper, we will use the Gaussian hypergeometric function ${_2}F_1$ and the functions ${_1}F_2$ and ${_0}F_1.$

\end{enumerate}

\subsection*{Acknowledgments}
\mbox{}The authors are grateful to the reviewer for careful reading of the paper and helpful comments, and to Adam Nowak for literature references.

M.\ Kwaśnicki was supported by the National Science Centre, Poland, grant no.\@ 2023/49/B/ST1/04303. B. Wr{\'o}bel was supported by the National Science Centre, Poland grant no.\@ 2022/46/E/ST1/00036.

This research was funded in whole or in part by National Science Centre, Poland, research projects
2023/49/B/ST1/04303 and 2022/46/E/ST1/00036. For the purpose of Open Access, the authors have applied a
CC-BY public copyright licence to any Author Accepted Manuscript (AAM) version arising from this submission.

\section{Formulas for the kernel \texorpdfstring{$b_k$}{bk} and its Fourier transform via special functions}
\label{sec: form}

Our first goal in this section is to derive two formulas for $\widehat{b}_k,$ one in terms of the Bessel function $J_\nu$ and another in terms of the generalized hypergeometric function $_1F_2.$ The second formula \eqref{eq: mk 2} is  not strictly needed in  this paper; however, we state it for potential future applications.
\begin{proposition}
\label{pro: mk}
The radial profile $m_k$ of the multiplier $\widehat{b}_k$ can be expressed by 
 \begin{equation}
\label{eq: mk 1}
m_k(r)=\frac{2^{d/2} \Gamma(\frac{d + k}{2})}{\Gamma(\frac{k}{2})} \int_{2\pi r}^\infty t^{-d/2} J_{d/2 + k - 1}(t) \, dt,\qquad r>0,
\end{equation}
and $m_k(0)=1.$ Furthermore
\begin{equation}
\label{eq: mk 2}
 m_k(r) = 1 - \frac{\Gamma(\frac{d + k}{2})}{\Gamma(\frac{d}{2} + k) \Gamma(\frac{k}{2} + 1)} \, (\pi r)^k {_1F_2}(\tfrac{k}{2}; \tfrac{d}{2} + k, \tfrac{k}{2} + 1; -(\pi r)^2),\qquad r>0.
 \end{equation}
\end{proposition}
\begin{proof}
We first justify \eqref{eq: mk 1}. Fix $P\in \mH_k$ and denote by $\rho_P^1$ the multiplier symbol of the operator $R_P^1$ given by \eqref{eq:R}. Let $\varphi(r) = \gamma_{k,d}r^{-d-k}\ind_{r > 1}$ be the radial profile of $\gamma_{k,d}|x|^{-d-k}\ind{_{|x|>1}}$. Then we have
\[
\rho_P^1(\xi)=\mF_{d}(P(\cdot)\varphi(|\cdot|))(\xi).
\]
 Using Bochner's relation for $P\in \mH_k$ (see, e.g., \cite[Corollary p.72]{stein}), together with a standard approximation argument, we see that
\[
\rho_P^1(\xi)=(-i)^k P(\xi)\Phi(\abs{\xi}),
\]
where $\Phi$ is defined by
\begin{equation}
\label{eq: defPhi}
\mF_{d+2k}(\varphi(|\cdot|))(\eta)=\Phi(|\eta|),\qquad \eta\in \R^{d+2k}.
\end{equation}
Since 
\[
\rho_P^1(\xi)=(-i)^k \frac{P(\xi)}{|\xi|^k}(|\xi|^k\Phi(\abs{\xi}))
\]
we have
\[
\mF_d[R_P^1 f](\xi)=\mF_d[R_P f](\xi)|\xi|^k\Phi(\abs{\xi}),
\]
so that
$|\xi|^k\Phi(
\abs{\xi})=\widehat{b}_k(\xi)$ and thus $m_k(r)=r^k\Phi(r),$ $r>0.$

To prove \eqref{eq: mk 1} we come back to \eqref{eq: defPhi} and write the Fourier transform of the radial function $\varphi(|y|)$ on $\R^{d+2k}$ in terms of Bessel functions (the Hankel transform). Applying \cite[Section B.5]{grafakos}, we see that
\[
 m_k(r) = r^k\frac{2\pi \gamma_{k,d}}{r^{n/2-1}} \int_1^\infty t^{-d-k} J_{n/2 - 1}(2\pi tr) t^{n/2} \, dt, 
\]
where $n=d+2k.$ Changing variables and recalling the definition of $\gamma_{k,d}$, we reach \eqref{eq: mk 1}. Since $b_k\in L^1(\R^d)$, we know that  $m_k$ is a continuous function on $[0,\infty).$ Thus the equation $m_k(0)=1$ follows from \eqref{eq: mk 1} and \href{https://dlmf.nist.gov/10.22.E43}{equation~10.22.43} in~\cite{dlmf}. 

It remains to prove \eqref{eq: mk 2}.  Let $\tm(r)=m_k(r/(2\pi))$. From \eqref{eq: mk 1} we have 
\[
    \tm'(r) = -\frac{2^{d/2} \Gamma(\frac{d + k}{2})}{\Gamma(\frac{k}{2})}\, r^{-d/2} J_{d/2 + k - 1}(r) ,
\]
and thus, by \href{https://dlmf.nist.gov/10.16.E9}{equation~10.16.9} in~\cite{dlmf},
\[
 \tm'(r) = -\frac{\Gamma(\frac{d + k}{2})}{\Gamma(\frac{d}{2} + k) \Gamma(\frac{k}{2})}  \, (\tfrac{r}{2})^{k - 1} {_0F_1}(\tfrac{d}{2} + k; -(\tfrac{r}{2})^2),
\]
where $_0F_1$ denotes the generalized hypergeometric function. Furthermore, since 
\[
 \int_0^\infty \tm'(r) \, dr = -\tm(0) =  -1 , 
\]
we have
\begin{equation*}
 \tm(r) = 1 + \int_0^r \tm'(t) \, dt = 1- \frac{\Gamma(\frac{d + k}{2})}{\Gamma(\frac{d}{2} + k) \Gamma(\frac{k}{2})} \int_{0}^r (\tfrac{t}{2})^{k - 1} {_0F_1}(\tfrac{d}{2} + k; -(\tfrac{t}{2})^2) \, dt.
\end{equation*}
Using the change of variables $u= (t/r)^2$, we obtain
\begin{equation*}
\tm(r)=1-\frac{\Gamma(\frac{d + k}{2})}{\Gamma(\frac{d}{2} + k)\Gamma(\frac{k}{2})}(\tfrac{r}{2})^{k}\int_{0}^1 u^{k/2 - 1} {_0F_1}(\tfrac{d}{2} + k; -(\tfrac{r}{2})^2u) \, du.
\end{equation*}
Thus, \href{https://dlmf.nist.gov/16.5.E2}{equation~16.5.2} in~\cite{dlmf} applied with $a_0=k/2,$ $b_0=k/2+1$, $b_1=d/2+k$ and $z=-(\tfrac{r}{2})^2$ gives
\[
    \tm(r)=1-\frac{\Gamma(\frac{d + k}{2})}{\Gamma(\frac{d}{2} + k)\Gamma(\frac{k}{2}+1)}(\tfrac{r}{2})^k {_1F_2}(\tfrac{k}{2}; \tfrac{d}{2} + k, \tfrac{k}{2} + 1; -(\tfrac{r}{2})^2).
\]
Finally, coming back to $m_k(r)=\tm(2\pi r)$ we reach \eqref{eq: mk 2}. The proof of Proposition \ref{pro: mk} is thus completed.

\end{proof}

Using  Proposition \ref{pro: mk} we now give an expression for $b_k.$ 
\begin{proposition}
\label{pro: B_k}
    Let $B_k$ be the radial profile of $b_k.$ Then we have
\[
 B_k(r) = \begin{cases}
  \displaystyle
  \frac{(\Gamma(\frac{d + k}{2}))^2}{\pi^{d/2} \Gamma(\frac{d}{2} + 1) (\Gamma(\frac{k}{2}))^2} \, {_2F_1}(\tfrac{d + k}{2}, 1 - \tfrac{k}{2}; \tfrac{d}{2} + 1; r^2) & \text{if } \ r \in [0, 1), \\[1.2em]
  \displaystyle
  \frac{(\Gamma(\frac{d + k}{2}))^2}{\pi^{d/2 + 1} \Gamma(\frac{d}{2} + k)} \, \frac{1}{r^{d + k}} \, {_2F_1}(\tfrac{d + k}{2}, \tfrac{k}{2}; \tfrac{d}{2} + k; r^{-2}) \sin \tfrac{k \pi}{2} & \text{if } \ r \in (1, \infty).
 \end{cases}
\]
In particular, when $k$ is even we have $B_k(r) = 0$ for $r \in (1, \infty)$, and $B_k(r)$ is a polynomial of degree $k-2$ for $r \in [0, 1)$. However, no such simplification occurs when $k$ is odd.
\end{proposition}

\begin{proof}
We proceed similarly to \cite[Appendix 4.1]{LiuMeZhu}.
Using the expression for the Fourier transform on $\R^d$ of the radial function $m_k(r)$ from \cite[Section B.5]{grafakos}, followed by the change of variables $2\pi tr=s$ and the formula $(s^{d/2}J_{d/2}(s))'= s^{d/2} J_{d/2-1}(s)$ (see \href{https://dlmf.nist.gov/10.6.E6}{equation~10.6.6} in~\cite{dlmf}), we obtain
\begin{equation*}
    B_k(r)=\frac{2\pi}{r^{d/2-1}}\int_{0}^{\infty}m_k(t)J_{d/2-1}(2\pi tr)t^{d/2}\,dt=\frac{1}{(2\pi)^{d/2}r^{d}}\int_{0}^{\infty}m_k(\tfrac{s}{2\pi r}) (s^{d/2}J_{d/2}(s))' \, ds.
\end{equation*}
A repetition of the argument used to prove \cite[eq.\ (4.2)]{LiuMeZhu} shows that 
\[m_k(\tfrac{s}{2\pi r})=O(s^{-(d+1)/2}),\qquad s\to \infty.\]
Thus, using integration by parts, $J_{d/2}(s)=O(s^{-1/2})$ and \eqref{eq: mk 1}, we obtain
\begin{align*}
    B_k(r)&=\frac{1}{(2\pi)^{d/2}r^{d}}\left[m_k(\tfrac{s}{2\pi r})s^{d/2}J_{d/2}(s)\right]_{s=0}^{s=\infty}-\frac{1}{(2\pi)^{d/2}r^{d}}\int_{0}^{\infty}\frac{d}{ds}\left(m_k(\tfrac{s}{2\pi r})\right) s^{d/2}J_{d/2}(s) \, ds\\
    &=\frac{\Gamma(\frac{d + k}{2})}{\pi^{d/2}\Gamma(\frac{k}{2})r^{d/2+1}}\int_{0}^{\infty} J_{d/2+k-1}(\tfrac{s}{r})J_{d/2}(s) \, ds
\end{align*}

We  now consider  the cases $r>1$ and $r<1$ separately. If $r<1$, we use \href{https://dlmf.nist.gov/10.22.E56}{equation~10.22.56} in~\cite{dlmf} (or equation~2.12.31.1 in~\cite{prudnikov}) with  $\lambda=0,$ $\mu=d/2,$ $\nu=d/2+k-1,$ and $a=1,$ $b=1/r;$ note that in this equation ${\bf F}(a,b;c;z)={_2}F_1(a,b;c;z)/\Gamma(c)$. This leads to
\begin{equation}
\label{eq: Bk1}
  B_k(r)=\frac{\Gamma(\frac{d + k}{2})^2}{\pi^{d/2}\Gamma(\frac{k}{2})^2}\frac{{_2}F_1(\tfrac{d + k}{2}, 1-\tfrac{k}{2}; \tfrac{d}{2} + 1; r^{2})}{\Gamma(\frac{d}{2}+1)}.
\end{equation}
If $r>1$ we use again \href{https://dlmf.nist.gov/10.22.E56}{equation~10.22.56} in~\cite{dlmf}, this time with $\lambda=0,$ $\mu=d/2+k-1,$ $\nu=d/2,$ and $a=1/r,$ $b=1.$ This gives
\begin{align*}
  B_k(r)&=\frac{\Gamma(\frac{d + k}{2})^2}{\pi^{d/2}\Gamma(\frac{k}{2})\Gamma(-\frac{k}{2}+1)r^{d+k}}\frac{{_2}F_1(\tfrac{d + k}{2}, \tfrac{k}{2}; \tfrac{d}{2} + k; r^{-2})}{\Gamma(\frac{d}{2}+k)},
\end{align*}
where it is understood that $\Gamma(-\frac{k}{2} + 1) = \infty$ and $B_k(r) = 0$ if $k$ is even. When $k$ is odd we use \href{https://dlmf.nist.gov/5.5.E3}{equation~5.5.3} in \cite{dlmf} and obtain
\begin{equation}
\label{eq: Bk2}
  B_k(r)=\frac{\Gamma(\frac{d + k}{2})^2}{\pi^{d/2+1}r^{d+k}}\frac{{_2}F_1(\tfrac{d + k}{2}, \tfrac{k}{2}; \tfrac{d}{2} + k; r^{-2})}{\Gamma(\frac{d}{2}+k)}\sin\tfrac{k\pi}{2}.
\end{equation}
Note that \eqref{eq: Bk2} holds also for even $k$ in which case both sides are zero. Finally, using \eqref{eq: Bk1}, \eqref{eq: Bk2} we complete the proof of Proposition \ref{pro: B_k}.

\end{proof}

\section{Kernel estimates --- proofs of Theorem \ref{thm: bkL1} and Corollary \ref{cor: bkL1}}

\label{sec: kere}

\subsection{Proof of Theorem \ref{thm: bkL1} --- sign change of the kernel.}
As we already mentioned, when $k=1$, the non-negativity of $b_k$ follows from \cite{LiuMeZhu}, while for $k=2$ it is contained e.g.\ in \cite[p.\ 427]{verdera conf}. 

We shall prove that for $k\ge 3$, the radial profile $B_k$ of $b_k$ changes sign inside the interval $(0,1)$. Denoting
\[
l=\tfrac{d + k}{2},\qquad m=1 - \tfrac{k}{2},\qquad n=l+m=d/2+1
\]
and using Proposition \ref{pro: B_k} we see that it is enough to justify that ${_2F_1}(l, m; n; x)$ changes sign in $(0,1).$ We will achieve this by showing that ${_2F_1}(l, m; n; x)$ has a simple zero in $(0,1)$.

 We apply the formula for the number of zeroes of ${_2}F_1(l,m,n;x)$ from \cite[p.\ 586, eq.\ (18)]{Klein1}. For $u\in \R$ we let $E(u)$ be the largest integer 
which is smaller than $u.$ Then, according to the aforementioned formula, the number of zeroes of ${_2}F_1(l,m,n;x)$ in the interval $(0,1)$ is equal to
\[
E\left(\frac{|l-m|-|1-n|-|n-l-m|+1}{2}\right)=E(k/2)
\]
and this is larger than or equal to $E(3/2)=1.$ Thus, there is a zero inside $(0,1),$ call it $x_0$. To show that $x_0$ is simple, we note that the hypergeometric function ${_2F_1}(a, b; c; x)$ satisfies the  \href{https://dlmf.nist.gov/15.10.E1}{hypergeometric equation~15.10.1} in~\cite{dlmf}, which is non-singular in $(0,1).$ By the uniqueness of solutions to such ODE's, the existence of a higher-order zero would imply that ${_2}F_1(l,m,n;x)$ is identically zero. This shows that $x_0$ is indeed a simple zero and ${_2}F_1(l,m,n;x)$ does change sign around it.

\subsection{Proof of \texorpdfstring{\eqref{eq: bkL1 lim}}{(\ref*{eq: bkL1 lim})} in Theorem \texorpdfstring{\ref{thm: bkL1}}{\ref*{thm: bkL1}} --- odd \texorpdfstring{$k$}{k}}
It turns out that in the case of odd $k$, the proof of \eqref{eq: bkL1 lim} is easier. When $k$ is odd and $r \in (1, \infty)$, we have
\[
|\sph^{d - 1}| r^{d - 1} |B_k(r)| = \frac{2 (\Gamma(\frac{d + k}{2}))^2}{\pi \Gamma(\frac{d}{2}) \Gamma(\frac{d}{2} + k)} \, \frac{1}{r^{k + 1}} \, {_2F_1}(\tfrac{d + k}{2}, \tfrac{k}{2}; \tfrac{d}{2} + k; r^{-2}) .
\]
By the definition of Gauss's hypergeometric function (\href{https://dlmf.nist.gov/15.2#E1}{equation~15.2.1} in~\cite{dlmf}), we have
\[
 {_2F_1}(\tfrac{d + k}{2}, \tfrac{k}{2}; \tfrac{d}{2} + k; r^{-2}) = \sum_{n = 0}^\infty \frac{(\frac{d + k}{2})^{\overline{n}} (\frac{k}{2})^{\overline{n}}}{n! (\frac{d}{2} + k)^{\overline{n}}} \, \frac{1}{r^{2 n}} \, ,
\]
where $a^{\overline{n}} = a (a + 1) \dots (a + n - 1)$ is the rising factorial. The coefficients of the above hypergeometric series are positive, increasing functions of $d$, and they converge to a finite limit
\[
 \lim_{d \to \infty} \frac{(\frac{d + k}{2})^{\overline{n}} (\frac{k}{2})^{\overline{n}}}{n! (\frac{d}{2} + k)^{\overline{n}}} = \frac{(\frac{k}{2})^{\overline{n}}}{n!} .
\]
Hence, by the monotone convergence theorem and \href{https://dlmf.nist.gov/4.6.E7}{equation~4.6.7} in~\cite{dlmf}, ${_2F_1}(\tfrac{d + k}{2}, \tfrac{k}{2}; \tfrac{d}{2} + k; r^{-2})$ increases with $d$ to a finite limit
\[
 \lim_{d \to \infty} {_2F_1}(\tfrac{d + k}{2}, \tfrac{k}{2}; \tfrac{d}{2} + k; r^{-2}) = \sum_{n = 0}^\infty \frac{(\frac{k}{2})^{\overline{n}}}{n!} \, \frac{1}{r^{2 n}} = (1 - r^{-2})^{-k / 2} .
\]
By another application of the monotone convergence theorem we have
\[
 \lim_{d \to \infty} \int_1^\infty \frac{1}{r^{k + 1}} \, {_2F_1}(\tfrac{d + k}{2}, \tfrac{k}{2}; \tfrac{d}{2} + k; r^{-2}) \, dr = \int_1^\infty \frac{(1 - r^{-2})^{-k / 2}}{r^{k + 1}} \, dr ,
\]
and the right-hand side is infinite if $k \ge 3$ due to a nonintegrability at $r \to 1^+$. Furthermore, by Stirling's approximation $\Gamma(a) \sim \sqrt{2 \pi} \, a^{a - 1/2} e^{-a}$ as $a \to \infty$,
\begin{align*}
 \lim_{d \to \infty} \frac{2 (\Gamma(\frac{d + k}{2}))^2}{\pi \Gamma(\frac{d}{2}) \Gamma(\frac{d}{2} + k)} & = \lim_{d \to \infty} \frac{2 (\frac{d + k}{2})^{d + k - 1}}{\pi (\frac{d}{2})^{(d - 1) / 2} (\frac{d}{2} + k)^{(d - 1) / 2 + k}} \\
 &  = \lim_{d \to \infty} \frac{2 (1 + \frac{k}{d})^{d + k - 1}}{\pi (1 + \frac{2 k}{d})^{(d - 1) / 2 + k}} = \frac{2 e^k}{\pi e^k} = \frac{2}{\pi}
\end{align*}
and altogether, integrating in polar coordinates we obtain
\[
\liminf_{d\to \infty}\|b_k\|_{L^1(\R^d)}\ge \liminf_{d\to \infty} \frac{2 (\Gamma(\frac{d + k}{2}))^2}{\pi \Gamma(\frac{d}{2}) \Gamma(\frac{d}{2} + k)} \, \int_{1}^{\infty}\frac{1}{r^{k + 1}} \, {_2F_1}(\tfrac{d + k}{2}, \tfrac{k}{2}; \tfrac{d}{2} + k; r^{-2}) \,dr=\infty.
\]


\subsection{Proof of \texorpdfstring{\eqref{eq: bkL1 lim}}{(\ref*{eq: bkL1 lim})} in Theorem \texorpdfstring{\ref{thm: bkL1}}{\ref*{thm: bkL1}} --- even \texorpdfstring{$k$}{k}}
For even $k$ we have $B_k(r) = 0$ when $r \in (1, \infty)$, and a more careful analysis of the behaviour of $B_k(r)$ for $r \in [0, 1)$ is necessary. Throughout the proof, the symbol $O$ contains an implicit constant that depends on $k.$

In the integral of $|\sph^{d - 1}| r^{d - 1} |B_k(r)|$ over $r \in [0, 1)$, the majority of mass accumulates near $r = 1$. It turns out that in order to have integrands converging to a non-trivial limit as $d \to \infty$, the substitution $r^d = e^{-s}$ is the right one. With this change of variables,
\[
\|b_k\|_{L^1(\R^d)}= |\sph^{d - 1}| \int_0^1 r^{d - 1} |B_k(r)| \, dr = \frac{|\sph^{d - 1}|}{d} \int_0^\infty |B_k(e^{-s / d})| e^{-s} \, ds .
\]
For even $k$ and $s \in [0, \infty)$, we have
\begin{equation}
\label{eq:kernel}
 \frac{|\sph^{d - 1}|}{d} \, B_k(e^{-s / d}) = \frac{(\Gamma(\frac{d + k}{2}))^2}{(\Gamma(\frac{d}{2} + 1) \Gamma(\frac{k}{2}))^2} \, {_2F_1}(\tfrac{d + k}{2}, 1 - \tfrac{k}{2}; \tfrac{d}{2} + 1; e^{-2 s / d}) .
\end{equation}
We claim that
\begin{equation}
\label{eq:stirling}
 \frac{(\Gamma(\frac{d + k}{2}))^2}{(\Gamma(\frac{d}{2} + 1) \Gamma(\frac{k}{2}))^2} = \frac{(\frac{d}{2})^{k - 2}}{(\Gamma(\frac{k}{2}))^2} \, (1 + O(d^{-1}))
\end{equation}
as $d \to \infty$. Indeed, by Stirling's approximation $\Gamma(a) =  \sqrt{2 \pi} \, a^{a - 1/2} e^{-a} (1 + O(a^{-1}))$ as $a \to \infty$,  we have
\begin{align*}
 \frac{(\Gamma(\frac{d + k}{2}))^2}{(\Gamma(\frac{d}{2} + 1))^2 (\frac{d}{2})^{k - 2}} & = \frac{(\Gamma(\frac{d + k}{2}))^2}{(\Gamma(\frac{d}{2}))^2 (\frac{d}{2})^k} \\
 & = \frac{(\frac{d + k}{2})^{d + k - 1}}{(\frac{d}{2})^k (\frac{d}{2})^{d - 1} e^k} \, (1 + O(d^{-1})) \\
 & = \frac{(1 + \frac{k}{d})^{d + k - 1}}{e^k} \, (1 + O(d^{-1})) \\
 & = e^{(d + k - 1) \log(1 + k/d) - k} \, (1 + O(d^{-1})) \\
 & = 1 + O(d^{-1}) ,
\end{align*}
where the last identity follows from Taylor's expansion of the logarithm.
This proves our claim \eqref{eq:stirling}.

The other factor in~\eqref{eq:kernel} is, however, more complicated. By the definition of the hypergeometric function (\href{https://dlmf.nist.gov/16.2#E1}{equation~16.2.1} in~\cite{dlmf}),
\[
 {_2F_1}(\tfrac{d + k}{2}, 1 - \tfrac{k}{2}; \tfrac{d}{2} + 1; e^{-2 s / d}) = \sum_{j = 0}^{k / 2 - 1} \frac{(\frac{d + k}{2})^{\overline{j}} (1 - \frac{k}{2})^{\overline{j}}}{j! (\tfrac{d}{2} + 1)^{\overline{j}}} \, e^{-2 j s / d} .
\]
Furthermore,
\[
 \frac{(\frac{d + k}{2})^{\overline{j}} (1 - \frac{k}{2})^{\overline{j}}}{j! (\tfrac{d}{2} + 1)^{\overline{j}}} = (-1)^j \binom{\frac{k}{2} - 1}{j} \frac{(\frac{d + k}{2})^{\overline{j}}}{(\tfrac{d}{2} + 1)^{\overline{j}}} = (-1)^j \binom{\frac{k}{2} - 1}{j} \frac{(\frac{d}{2} + 1 + j)^{\overline{k / 2 - 1}}}{(\tfrac{d}{2} + 1)^{\overline{k / 2 - 1}}} \, ,
\]
and hence
\begin{equation}
\label{eq:delta:pre}
 {_2F_1}(\tfrac{d + k}{2}, 1 - \tfrac{k}{2}; \tfrac{d}{2} + 1; e^{-2 s / d}) = \frac{1}{(\tfrac{d}{2} + 1)^{\overline{k / 2 - 1}}} \sum_{j = 0}^{k / 2 - 1} (-1)^j \binom{\frac{k}{2} - 1}{j} \lambda_j ,
\end{equation}
where
\[
 \lambda_j = (\tfrac{d}{2} + 1 + j)^{\overline{k / 2 - 1}} e^{-2 j s / d} .
\]

Before we continue, let us introduce some notation. We denote the forward difference of a sequence $a=(a_n)_{n\in \N}$ by $(\Delta a)_n = a_{n + 1} - a_n$. The $m$th iterated difference $\Delta^m a$ satisfies
\[
 (-1)^m (\Delta^m a)_n = \sum_{j = 0}^m (-1)^j \binom{m}{j} a_{n + j} .
\]
We can thus rewrite~\eqref{eq:delta:pre} as
\begin{equation}
\label{eq:delta}
 {_2F_1}(\tfrac{d + k}{2}, 1 - \tfrac{k}{2}; \tfrac{d}{2} + 1; e^{-2 s / d}) = \frac{(-1)^{k / 2 - 1}}{(\tfrac{d}{2} + 1)^{\overline{k / 2 - 1}}} \, (\Delta^{k / 2 - 1} \lambda)_0 .
\end{equation}
The iterated difference on the right-hand side cannot be evaluated explicitly. However, we may find its asymptotic behaviour as $d \to \infty$ using Taylor's expansion
\[
 e^{-2 j s / d} = \sum_{n = 0}^{k / 2 - 1} \frac{(-1)^n j^n s^n}{n! (\frac{d}{2})^n} + O(d^{-k / 2}),\qquad s\ge 0.
\]
The implicit constant in the big $O$ above and in all the following big $O$ symbols in the proof depends on both $k$ and $s.$ This, however, will not impact the proof as we will be only interested in taking the limit as $d\to \infty.$

It follows that
\[
 \lambda_j = \sum_{n = 0}^{k / 2 - 1} \frac{(-1)^n j^n s^n (\tfrac{d}{2} + 1 + j)^{\overline{k / 2 - 1}}}{n! (\frac{d}{2})^n} + O(d^{-1}) .
\]
By the binomial theorem for rising factorials (Exercise~5.37 in~\cite{concrete-mathematics}),
\[
 \lambda_j = \sum_{n = 0}^{k / 2 - 1} \sum_{m = 0}^{k / 2 - 1} \binom{\frac{k}{2} - 1}{m} \frac{(-1)^n j^n s^n (\tfrac{d}{2} + 1)^{\overline{m}} j^{\overline{k / 2 - 1 - m}}}{n! (\frac{d}{2})^n} + O(d^{-1}) .
\]
The terms with $m < n$ can be absorbed into $O(d^{-1})$, and so
\[
 \lambda_j = \sum_{n = 0}^{k / 2 - 1} \sum_{m = n}^{k / 2 - 1} \binom{\frac{k}{2} - 1}{m} \frac{(-1)^n j^n s^n (\tfrac{d}{2} + 1)^{\overline{m}} j^{\overline{k / 2 - 1 - m}}}{n! (\frac{d}{2})^n} + O(d^{-1}) .
\]
Each term under the sum is a polynomial in $j$ of degree $\tfrac{k}{2} - 1 + n - m$, which does not exceed $\frac{k}{2} - 1$. Recall that the iterated difference of order $\frac{k}{2} - 1$ applied to a polynomial of degree less than $\frac{k}{2} - 1$ is zero, so in the evaluation of $(\Delta^{k / 2 - 1} \lambda)_0$, all terms corresponding to $m > n$ disappear. Furthermore, the iterated difference of order $\tfrac{k}{2} - 1$ applied to the monomial $j^{k/2 - 1}$ is equal to $(\frac{k}{2} - 1)!$, and $j^n j^{\overline{k/2 - 1 - n}}$ is the sum of $j^{k/2 - 1}$ and a polynomial of degree less than $\frac{k}{2} - 1$. Altogether we find that
\[
 (\Delta^{k / 2 - 1} \lambda)_0 = \sum_{n = 0}^{k / 2 - 1} \binom{\frac{k}{2} - 1}{n} \frac{(-1)^n s^n (\tfrac{d}{2} + 1)^{\overline{n}} (\frac{k}{2} - 1)!}{n! (\frac{d}{2})^n} + O(d^{-1}) .
\]

Expanding the rising factorial $(\tfrac{d}{2} + 1)^{\overline{n}}$ and absorbing each term with an exponent of $d$ less than $n$ in $O(d^{-1})$, we obtain
\[
( \Delta^{k / 2 - 1} \lambda)_0 = \sum_{n = 0}^{k / 2 - 1} \binom{\frac{k}{2} - 1}{n} \frac{(-1)^n s^n (\frac{k}{2} - 1)!}{n!} + O(d^{-1}) .
\]
Substituting this expression into~\eqref{eq:delta} yields
\[
 {_2F_1}(\tfrac{d + k}{2}, 1 - \tfrac{k}{2}; \tfrac{d}{2} + 1; e^{-2 s / d}) = \frac{(-1)^{k / 2 - 1}}{(\tfrac{d}{2} + 1)^{\overline{k / 2 - 1}}} \biggl( \sum_{n = 0}^{k / 2 - 1} \binom{\frac{k}{2} - 1}{n} \frac{(-1)^n s^n (\frac{k}{2} - 1)!}{n!} + O(d^{-1}) \biggr) .
\]
Finally, $1 / (\tfrac{d}{2} + 1)^{\overline{k / 2 - 1}} = (\tfrac{d}{2})^{1 - k / 2} (1 + O(d^{-1}))$. Together with~\eqref{eq:stirling}, this allows us to rewrite~\eqref{eq:kernel} as
\begin{align*}
 &\frac{|\sph^{d - 1}|}{d} \, B_k(e^{-s / d}) = \frac{(-1)^{k / 2 - 1} (\frac{d}{2})^{k/2 - 1}}{(\Gamma(\frac{k}{2}))^2} \biggl( \sum_{n = 0}^{k / 2 - 1} \binom{\frac{k}{2} - 1}{n} \frac{(-1)^n s^n (\frac{k}{2} - 1)!}{n!} + O(d^{-1}) \biggr)\\
 & = \frac{(-1)^{k / 2 - 1}(\frac{k}{2}-1)! (\frac{d}{2})^{k/2 - 1}}{(\Gamma(\frac{k}{2}))^2} (L_{k/2-1}(s)+O(d^{-1})),
\end{align*}
where $L_{k/2-1}(s)$  is the Laguerre polynomial of degree $k/2-1$ (see \href{https://dlmf.nist.gov/18.5.E12}{equation~18.5.12} in~\cite{dlmf}).

Here, of course, the implicit constant in $O(d^{-1})$ depends on both $k$ and $s$. It follows that
\[
\|b_k\|_{L^1(\R^d)} = \frac{(\tfrac{d}{2})^{k/2 - 1}}{(\frac{k}{2} - 1)!} \int_0^\infty |L_{k/2 - 1}(s) + O(d^{-1})| e^{-s} \, ds ,
\]
and so, by Fatou's lemma,
\begin{equation}
\label{eq: bkLag}
\liminf_{d \to \infty} \frac{\|b_k\|_{L^1(\R^d)}}{(\frac{d}{2})^{k / 2 - 1}} \ge \frac{1}{(\frac{k}{2} - 1)!} \int_0^\infty |L_{k/2 - 1}(s)| e^{-s} \, ds > 0 .
\end{equation}
In particular, $\|b_k\|_{L^1(\R^d)}$ is unbounded as $d \to \infty$ for every even $k \ge 4$.

\subsection{Proof of  Corollary \ref{cor: bkL1}}
From \eqref{eq:fact} and the assumptions, we have
\[
|M_k(R_k f)(x)|\le C_{k,d} |\mA_d(R_k f)(x)|
\]
for all Schwartz functions $f$ on $\R^d$ and a.e.\ $x\in \R^d.$ Using \eqref{eq:m}, it is easy to see that $\mD= \{ R_k f \colon f\colon \R^d\to \C  \textrm{ is Schwartz}\}$ is a dense subset of $L^2(\R^d).$ Take a general function $g\in L^2(\R^d)$ and assume that $g_n\in \mD$ converges to $g$ in  $L^2(\R^d).$ Using the continuity of $M_k$ and $\mA_d$ we know that $M_k(g_n)\to M_k(g)$ and $\mA_d(g_n)\to \mA_d(g),$ the convergence being in $L^2(\R^d).$ Passing to a subsequence, we may assume that the convergence also holds almost everywhere. Furthermore, since $g_n\in \mD$ we have
\[
|M_k(g_n)(x)|\le C_{k,d} |\mA_d(g_n)(x)|,\qquad x\textrm{-a.e.}
\]
Hence, taking $n\to \infty$ we see that
\begin{equation}
\label{eq: MkMa}
|M_k(g)(x)|\le C_{k,d} |\mA_d(g)(x)|,\qquad x\textrm{-a.e.},
\end{equation}
for $g\in L^2(\R^d).$ Denote by $A(p)$ the norm of $\mA_d$ as an operator on $L^p(\R^d),$ $p\in [2,\infty).$ Since $L^2(\R^d)$ is dense in $L^p(\R^d)$ using \eqref{eq: MkMa} we obtain
\[
\|M_k f\|_{L^p(\R^d)}\le C_{k,d} A(p) \|f\|_{L^p(\R^d)},\qquad f\in L^p(\R^d).
\]
 Note that because $\mA_d$ is an $L^{\infty}(\R^d)$ contraction, an explicit version of the Marcinkiewicz interpolation theorem, see e.g. \cite[Theorem 1.3.2]{grafakos}, implies that $\limsup_{p\to\infty} A(p) \le 3.$

Now, since $M_k$ is a convolution operator, we also see that
\[
\|M_k f\|_{L^p(\R^d)}\le 3C_{k,d} \|f\|_{L^p(\R^d)},\qquad f\in L^p(\R^d),
\]
for $p\to 1^+,$ and, consequently,
\[
\|b_k\|_{L^1(\R^d)}=\|M_k \|_{L^1(\R^d)\to L^1(\R^d)}\le 3C_{k,d}.
\]
Finally, Theorem \ref{thm: bkL1} shows that $C_{k,d}\to \infty$ as $d\to \infty,$ completing the proof.

%
%

\section{\texorpdfstring{$L^2$}{L2} estimates --- proofs of Theorem \ref{thm: fbk} and Corollary \ref{cor: gsir}}
\label{sec: L2e}


\subsection{Proof of Theorem \ref{thm: fbk}}
The case $k=1$ follows from \cite[Corollary 1.3]{LiuMeZhu}, while the case $k=2$ is a consequence of the formula $b_2=\frac{1}{|B|}\ind_{B},$ cf.\ \cite[p.\ 427]{verdera conf}, which implies that $\|b_2\|_{L^1(\R^d)} \le1.$ Hence, in the proof, we focus on $k\ge 3.$

Note that to prove Theorem \ref{thm: fbk}, it is enough to show that the radial profile $m_k(|\xi|)=\widehat{b}_k(\xi)$ satisfies $|m_k(r)|\le 1.$ Then \eqref{eq: RktRL2} easily follows from the factorization \eqref{eq:fact} and Plancherel's theorem. Recalling the abbreviation $\tm(r)=m_k(r/(2\pi))$, our task boils down to verifying
\begin{equation}
\label{eq: mk<1}
|\tm(r)|\le 1,\qquad r>0.
\end{equation}

The estimate \eqref{eq: mk<1} will be deduced from Proposition \ref{pro: mk} together with an oscillatory estimate for integrals of Bessel functions from  \cite{lms}, which we now describe. Let $\nu > \tfrac{1}{2}$ and $0 \le \alpha < \nu + \tfrac{3}{2}$. Denote by $j_{\nu, n}$ ($n = 1, 2, \ldots$) the $n$th zero of $J_\nu$ on $(0, \infty)$, and let $j_{\nu, 0} = 0$. By Theorem~5.2 in~\cite{lms} (with $W(x) = x^{-\alpha}$ and $\lambda = 1$), the sequence
\[
 a_{n}:= (-1)^{n} \int_{j_{\nu, n}}^{j_{\nu, n + 1}} t^{1/2-\alpha} J_\nu(t) \, dt
\]
is completely monotone: its $\ell$th iterated differences satisfy $(-1)^\ell (\Delta^{\ell} a)_n \ge 0$ for $n = 0, 1, \ldots$ and $\ell = 0, 1, \ldots$\, Furthermore, Theorem~6.1 in~\cite{lms} states that
\begin{equation}
\label{eq:j:est}
 \frac{a_0}{2} < \int_0^\infty t^{1/2-\alpha} J_\nu(t) \, dt < a_0 .
\end{equation}
As we shall see below, this is exactly what is needed for \eqref{eq: mk<1}.

Recall that by Proposition \ref{pro: mk}
\[
  \tm(r) = C \int_r^\infty t^{1/2 - \alpha} J_\nu(t) \, dt.
\]
 with $\nu = \frac{d}{2} + k - 1$, $\alpha = \frac{d + 1}{2}$ and $C = 2^{d/2} \Gamma(\frac{d + k}{2}) / \Gamma(\frac{k}{2})$.
Note that for $k\ge 3$ we have $\nu \ge 2$   and $0 < \alpha \le \nu $.  Hence we may apply the results of~\cite{lms} listed in the previous paragraph. Since for $r \in [j_{\nu, n}, j_{\nu, n+ 1}]$, we have
\[
 (-1)^n \tm'(r) = (-1)^{n + 1} r^{1/2-\alpha}J_{d/2 + k - 1}(r) \le 0 
\]
it follows that $\tm$ is monotone on this interval, and therefore
\begin{equation}
\label{eq:m:r}
 |\tm(r)| \le \max \{|\tm(j_{\nu, n})|, |\tm(j_{\nu, n + 1})|\}
\end{equation}
for $r \in [j_{\nu, n}, j_{\nu, n + 1}]$. It remains to estimate $\tm(j_{\nu, n})$.

We have
\[
 \tm(j_{\nu, n}) = C \sum_{\ell = n}^\infty \int_{j_{\nu, \ell}}^{j_{\nu, \ell + 1}} t^{1/2 - \alpha} J_\nu(t) \, dt = C \sum_{\ell = n}^\infty (-1)^\ell a_{\ell}.
\]
Since $a_n$ is completely monotone, the above sum is the tail of an alternating series. It follows that
\[
 \tm(j_{\nu, 0}) \ge \tm(j_{\nu, 2}) \ge \tm(j_{\nu, 4}) \ge \ldots \ge 0 \ge \ldots \ge \tm(j_{\nu, 5}) \ge \tm(j_{\nu, 3}) \ge \tm(j_{\nu, 1}) .
\]
Furthermore,
\[
 \tm(j_{\nu, 0}) - \tm(j_{\nu, 1}) = C \int_{j_{\nu, 0}}^{j_{\nu, 1}} t^{-\alpha} J_\nu(t) \, dt = C a_0 ,
\]
and by~\eqref{eq:j:est},
\[
 C a_0 \le 2 C \int_0^\infty t^{-\alpha} J_\nu(t) \, dt = 2 m(j_{\nu, 0}) .
\]
Combining this inequality with the previous equation, we find that
\[
 \tm(j_{\nu, 1}) = \tm(j_{\nu, 0}) - C a_0 \ge -\tm(j_{\nu, 0}) .
\]
Finally $\tm(j_{\nu, 0}) = \tm(0) = 1$ by Proposition \ref{pro: mk}, and so
\begin{equation}
\label{eq:m:j}
 1 = \tm(j_{\nu, 0}) \ge \tm(j_{\nu, 2}) \ge \tm(j_{\nu, 4}) \ge \ldots \ge 0 \ge \ldots \ge \tm(j_{\nu, 5}) \ge \tm(j_{\nu, 3}) \ge \tm(j_{\nu, 1}) \ge -1 .
\end{equation}
Inequalities~\eqref{eq:m:r} and~\eqref{eq:m:j} imply the desired estimate \eqref{eq: mk<1} and the proof of Theorem \ref{thm: fbk} is completed.

\subsection{Proof of Corollary \ref{cor: gsir}}
Since 
\[T_{\Omega}^t f(x)=T_{\Omega}^1 ( f(t\cdot))(t^{-1}x),\] it is easy to see that it suffices to consider $t=1.$ In the proof, we abbreviate $T^1=T_{\Omega}^1,$ $T=T_{\Omega}$ and
\[
	K(x) =K_{\Omega}(x)= \frac{\Omega(\frac{x}{\abs{x}})}{\abs{x}^d}, \quad x \in \R^d \setminus \{0\},
\]
and
\[
	K^1(x) = K(x) \ind_{[1,\infty)}(|x|).
\]

We know that $\Omega$ has an expansion in spherical harmonics, that is
\begin{equation} \label{eq:omega_exp}
	\Omega(x) = \sum_{k=1}^\infty \gamma_{k,d} P_k(x), \quad x \in S^{d-1}, \
\end{equation}
where $\gamma_{k,d}$ is defined in \eqref{eq:KP} and $P_k$ is a homogeneous harmonic polynomial of degree $k$. Note that there is no zero-order term in \eqref{eq:omega_exp} because $\int_{\mathbb{S}^d}\Omega(x)\,dx=0.$ Furthermore, the series in \eqref{eq:omega_exp} converges uniformly on $\mathbb{S}^{d-1}$.

Using \eqref{eq:omega_exp} we can express the operator $T^{1}$ as
\begin{equation} \label{eq:Tt}
	\begin{aligned}
		T^{1} f(x) = \int_{\abs{y} > 1} \frac{\Omega(\frac{y}{\abs{y}})}{\abs{y}^d} f(x-y) \, dy &= \sum_{k=1}^\infty \gamma_{k,d}\int_{\abs{y} > 1} \frac{ P_k(\frac{y}{\abs{y}})}{\abs{y}^d} f(x-y) \, dy \\
		&= \sum_{k=1}^\infty T_k^1 f(x).
	\end{aligned}
\end{equation}
Each of the operators $T_k^1$ is a truncated higher-order Riesz transform, namely, $T_k^1=R_{P_k}^1$ with  $R_{P_k}^1$ given by \eqref{eq:R}. As such, according to \eqref{eq:fact}, it can be factorized as $T_k^1 = M_k T_k,$ where $T_k=R_{P_k}$.

Let $m_k$ be the radial profile of the multiplier of the operator $M_k^1$ and let $f_0$ be the radial profile of $f$.  
Then, by Plancherel's theorem, \eqref{eq:Tt} and \eqref{eq:m}, we have
\begin{equation*}
	\norm{T^1 f}_{L^2(\R^d)}^2 = \norm{\widehat{T^1 f}}_{L^2(\R^d)}^2 
	= \int_{\R^d} \left( \sum_{k=1}^\infty m_k(\abs{\xi}) (-i)^kP_k(\tfrac{\xi}{|\xi|}) \widehat{f}_0(\abs{\xi}) \right)^2 \, d\xi.
\end{equation*}
Since the polynomials $P_k$ are orthogonal on $\mathbb{S}^{d-1}$, integrating in polar coordinates, we obtain
\begin{align*}
	\norm{T^1 f}_{L^2(\R^d)}^2&= \int_0^\infty \abs{\widehat{f}_0(r)}^2 \int_{\mathbb{S}^{d-1}} \left( \sum_{k=1}^\infty m_k(r) P_k(x) \right)^2 \, dx\, dr \\
	&= \int_0^\infty \abs{\widehat{f}_0(r)}^2 \int_{\mathbb{S}^{d-1}} \sum_{k=1}^\infty \abs{m_k(r)}^2 \abs{P_k(x)}^2 \, dx \, dr.
    \end{align*}
Finally, using Theorem \ref{thm: fbk}, orthogonality, and Plancherel's theorem we reach
    \begin{align*}
	\norm{T^1 f}_{L^2(\R^d)}^2&\leqslant \int_0^\infty \abs{\widehat{f}_0(r)}^2 \int_{\mathbb{S}^{d-1}} \sum_{k=1}^\infty \abs{P_k(x)}^2 \, dx \, dr\\
    &= \int_{\R^d}\left|\sum_{k=1}^\infty (-i)^kP_k(\xi/|\xi|) \widehat{f}(\xi)\right|^2\,d\xi = \norm{T f}_{L^2(\R^d)}^2.
\end{align*}
This completes the proof of Corollary \ref{cor: gsir}.

%
%

\newcommand{\doi}[1]{\href{https://doi.org/#1}{\texttt{\scriptsize DOI:#1}}}

%
%

\end{document}